\documentclass[11pt,reqno]{amsart}
\usepackage{a4wide}
\usepackage[utf8x]{inputenc}
\usepackage[T1]{fontenc}
\usepackage[english]{babel}
\usepackage{hyperref}
\usepackage{amsfonts,amsmath,amssymb,amsthm,amsrefs}
\usepackage{latexsym}
\usepackage{layout}
\usepackage{dsfont}
\usepackage{color}

\newtheorem{theorem}{Theorem}[section]

\newtheorem{lemma}{Lemma}[section]
\newtheorem{remark}{Remark}[section]

\newcommand{\p}{\mathbb{P}}
\newcommand{\e}{\mathbb{E}}

\newcommand{\reals}{\mathbb{R}}
\newcommand{\ind}{\mathbf{1}}

\newcommand{\Zq}[1]{Z_{p,q} \left(#1 \right)}
\newcommand{\Zqprime}[1]{Z_{p,q}^\prime \left(#1 \right)}
\newcommand{\Zqdoubleprime}[1]{Z_{p,q}^{\prime \prime} \left(#1 \right)}
\newcommand{\ruintime}{\sigma_p^\pi}

\newcommand{\admissible}{\Pi_{\beta,p}}

\allowdisplaybreaks

\begin{document}

\title[Impulse control with Parisian ruin for SNLPs]{A note on the optimal dividends problem with transaction costs in a spectrally negative Lévy model with Parisian ruin}

\author[J.-F. Renaud]{Jean-Fran\c{c}ois Renaud}
\address{D\'epartement de math\'ematiques, Universit\'e du Qu\'ebec \`a Montr\'eal (UQAM), 201 av.\ Pr\'esident-Kennedy, Montr\'eal (Qu\'ebec) H2X 3Y7, Canada}
\email{renaud.jf@uqam.ca}


\date{\today}


\begin{abstract}
In this note, merging ideas from \cite{loeffen_2009b} and \cite{renaud_2019}, we prove that an $(a,b)$-strategy maximizes dividend payments subject to fixed transaction costs in a spectrally negative Lévy model with Parisian ruin, as long as the tail of the Lévy measure is log-convex.
\end{abstract}

\maketitle


\section{Introduction}

Maximizing dividend payments is a trade-off between paying out (as early and) as much as possible while trying to avoid ruin for the possibility of more payments in the future. Classically, three versions of this problem have been studied: singular, absolutely continuous and impulse, which are related to the corresponding set of admissible control strategies. See, e.g., \cite{jeanblanc-shiryaev_1995} for an analysis of these three problems in a Brownian model. Later, these problems have been studied in more general models; in a spectrally negative Lévy model, they have been studied for example in \cites{avram-et-al_2007,loeffen_2008}, \cite{kyprianou-et-al_2012} and \cites{alvarez-rakkolainen_2009,loeffen_2009b, thonhauser-albrecher_2011}, respectively. See also \cite{albrecher-thonhauser_2009} for a survey on optimal dividend strategies.

In the above control problems, the termination time is the classical ruin time or, more precisely, the first time the controlled process goes below zero (or another pre-specified critical level). For more than a decade now, \textit{exotic} definitions of ruin have been studied. One of them is Parisian ruin, with exponential clocks or with deterministic delays. Naturally, the effect of a given type of Parisian ruin on the maximization of dividends has been considered. In a spectrally negative Lévy model, the singular control problem with Parisian ruin has been studied in \cite{czarna-palmowski_2014} and \cite{renaud_2019}, for deterministic delays and exponential delays respectively, while the absolutely continuous version with Parisian ruin is analyzed in \cite{locas-renaud_2023}.

In this note, we aim at \textit{completing the trilogy} by considering the impulse control problem in a spectrally negative Lévy model with Parisian ruin. Our main contribution is a solution to this problem, in the spirit of recent literature. More precisely, we give a fairly general condition on the Lévy measure of the underlying spectrally negative Lévy process under which a simple impulse control strategy is optimal. The condition is that the tail of the Lévy measure is log-convex. It turns out that this condition also improves the analysis originally provided in \cite{loeffen_2009b} (see \cite{loeffen-renaud_2010} for a discussion) when standard ruin is considered. It is also the condition imposed in the solution of the singular control problem in \cite{renaud_2019}. Our solution is based on techniques used and results obtained in both \cite{loeffen_2009b} and \cite{renaud_2019}. As a consequence, mathematical details will be kept at a minimum when standard arguments are involved.

\section{Problem formulation and verification lemma}

On a filtered probability space $\left( \Omega, \mathcal{F}, \left\lbrace \mathcal{F}_t, t \geq 0 \right\rbrace, \p \right)$, let $X=\left\lbrace X_t , t \geq 0 \right\rbrace$ be a spectrally negative Lévy process (SNLP) with Laplace exponent
\[
\theta \mapsto \psi(\theta) = \gamma \theta + \frac{1}{2} \sigma^2 \theta^2 + \int^{\infty}_0 \left( \mathrm{e}^{-\theta z} - 1 + \theta z \ind_{(0,1]}(z) \right) \nu(\mathrm{d}z) ,
\]
where $\gamma \in \reals$ and $\sigma \geq 0$, and where $\nu$ is a sigma-finite measure on $(0,\infty)$ satisfying
\begin{equation*}
\int^{\infty}_0 (1 \wedge x^2) \nu(\mathrm{d}x) < \infty .
\end{equation*}
This measure is called the L\'{e}vy measure of $X$. As $X$ is a (strong) Markov process, its law when starting from $X_0 = x$ will be denoted by $\p_x$ and the corresponding expectation by $\e_x$. If $X_0=0$, we will simply write $\p$ and $\e$. With this notation, we thus have $\psi(\theta)=\ln \left(\e \left[\mathrm{e}^{\theta X_1} \right] \right)$.

Let us model a cash process using $X$. We consider the optimization problem in which, at each dividend payment, a fixed transaction cost of $\beta>0$ is paid by the policyholders. Therefore, a dividend strategy $\pi$ is represented by a non-decreasing, left-continuous and adapted (thus predictable) stochastic process $L^\pi = \left\lbrace L^\pi_t , t \geq 0 \right\rbrace$ such that $L^\pi_0 = 0$ and
\begin{equation}\label{E:lump-sum}
L^\pi_t = \sum_{0 \leq s < t} \Delta L^\pi_s ,
\end{equation}
where $\Delta L^\pi_s = L^\pi_{s+} - L^\pi_s$, which represents the cumulative amount of dividends paid up to time $t$. The assumption in~\eqref{E:lump-sum} means that all dividend payments are lump sum dividend payments, which is in contrast with the singular version of the problem. Mathematically speaking, the process $L^\pi$ is a pure-jump process. In particular, $\sum_{0 \leq s < t} \ind_{\{\Delta L^\pi_s > 0 \}}$ represents the number of dividend payments made up to time $t$.

For a strategy $\pi$, the controlled cash process $U^\pi = \left\lbrace U^\pi_t , t \geq 0 \right\rbrace$ is given by $U^\pi_t = X_t - L^\pi_t$. In our control problem, the termination time is given by the time of Parisian ruin (with fixed rate $p>0$) defined by
\[
\ruintime = \inf \left\lbrace t>0 \colon t-g_t^\pi > \mathbf{e}_p^{g_t^\pi} \; \text{and} \; U^\pi_t < 0 \right\rbrace ,
\]
where $g_t^\pi = \sup \left\lbrace 0 \leq s \leq t \colon U^\pi_s \geq 0 \right\rbrace$, with $\mathbf{e}_p^{g_t^\pi}$ an independent random variable, following the exponential distribution with mean $1/p$, associated to the corresponding excursion below $0$.

A strategy $\pi$ is said to be admissible if $\Delta L^\pi_t \leq U^\pi_t$, for all $t < \ruintime$. In words, a dividend payment should not push the cash process down into the \textit{red zone} (below zero). As a consequence, no dividends can be paid when $U^\pi$ is below zero because the dividend process $L^\pi$ is assumed to be non-decreasing. Let us denote the set of admissible dividend strategies by $\admissible$. For a fixed discount rate $q \geq 0$, the performance function of an admissible dividend strategy $\pi \in \admissible$ is given by
\[
v_\pi (x) = \e_x \left[ \sum_{0 \leq t < \ruintime} \mathrm{e}^{-q t} \left( \Delta L^\pi_t - \beta \ind_{\{\Delta L^\pi_t > 0 \}} \right) \right] , \quad x \in \reals .
\]
Therefore, the value function $v_\ast$ of the problem is defined on $\reals$ by $v_\ast (x) = \sup_{\pi \in \admissible} v_\pi (x)$.

For the rest of the paper, we assume that the control problem parameters $\beta, p, q$ are fixed. Our goal is to obtain an analytical expression for $v_\ast$ by finding an optimal strategy $\pi_\ast \in \admissible$, i.e., such that $v_{\pi_\ast} (x) = v_\ast (x)$, for all $x \in \reals$, and for which we can compute the performance function.

\begin{remark}
Note that the performance function, which is now defined on the whole real line, can also be written as follows:
\[
v_\pi (x) = \e_x \left[ \int_0^{\ruintime} \mathrm{e}^{-q t} \mathrm{d} \left( L^\pi_t - \beta \sum_{0 \leq s < t} \ind_{\{\Delta L^\pi_s > 0 \}} \right) \right] .
\]
\end{remark}

Define the operator
\begin{equation}\label{eq:generator}
\Gamma v(x) = \gamma v^\prime (x)+\frac{\sigma^2}{2} v^{\prime \prime}(x) + \int_{0+}^\infty \left( v(x-z)-v(x)+v^\prime (x) z \ind_{(0,1]}(z) \right) \nu(\mathrm{d}z)
\end{equation}
acting on functions $v$, defined on $\reals$, such that $x \mapsto \Gamma v(x)$ is defined almost everywhere.

We say that a function $v$ is sufficiently smooth if it is continuously differentiable on $(0,\infty)$ when $X$ is of bounded variation and piecewise twice continuously differentiable when $X$ is of unbounded variation.

Here is a verification lemma for our maximization problem:
\begin{lemma}\label{verificationlemma}
Suppose that $\hat{\pi} \in \admissible$ is such that $v_{\hat{\pi}}$ is sufficiently smooth. If, for almost every $x \in \reals$,
\[
\Gamma v_{\hat{\pi}} (x) - \left( q+p \ind_{(-\infty,0)}(x) \right) v_{\hat{\pi}} (x) \leq 0
\]
and if, for all $y \geq x \geq 0$, 
\[
v_{\hat{\pi}}(y)-v_{\hat{\pi}}(x) \geq y-x-\beta ,
\]
then $\hat{\pi}$ is an optimal strategy for the control problem.
\end{lemma}
\begin{proof}
Set $w:=v_{\hat{\pi}}$ and let $\pi \in \admissible$ be an arbitrary admissible strategy. We will use the main ingredients in the proof of the verification lemma  in \cite{renaud_2019}. Let us first apply a change-of-variable formula or an Ito-type formula (depending on the level of smoothness of $w$) to the multidimensional process $\left( t, \int_0^t \ind_{(-\infty,0)}(U^\pi_r) \mathrm{d}r , U^\pi_t \right)$. Consequently, for $t > 0$, we can write
\begin{multline*}
\mathrm{e}^{-q t - p \int_0^t \ind_{(-\infty,0)}(U^\pi_r) \mathrm{d}r} w \left(U^\pi_t \right) - w \left(U^\pi_0 \right) \\
\leq \int_0^t \mathrm{e}^{-q s - p \int_0^s \ind_{(-\infty,0)}(U^\pi_r) \mathrm{d}r} \left[ \Gamma w \left(U^\pi_s \right) -q w \left(U^\pi_s \right) - p \ind_{(-\infty,0)} (U^\pi_s) w \left(U^\pi_s \right) \right] \mathrm{d}s \\
- \sum_{0 < s < t} \mathrm{e}^{-q s - p \int_0^s \ind_{(-\infty,0)}(U^\pi_r) \mathrm{d}r} \left(\Delta L^\pi_s - \beta \ind_{\{\Delta L^\pi_s > 0\}} \right) .
\end{multline*}
To obtain this last inequality, we used the fact that the jumps of $U^\pi$ can come from either $X$ or $L^\pi$ and that $L^\pi$ is a pure-jump process. Also, we used one of the assumptions yielding
\[
w \left(U^\pi_s - \Delta L^\pi_s  \right) - w \left(U^\pi_s \right) < - \Delta L^\pi_s + \beta
\]
at a time $s$ for which $\Delta L^\pi_s > 0$. 

Second, let us consider an independent (of the sigma-algebra $\mathcal{F}_\infty := \sigma \left( \cup_{i \geq 0} \mathcal{F}_i \right)$) Poisson process with intensity measure $p \, \mathrm{d}t$ and jump times $\left\lbrace T^p_i , i \geq 1 \right\rbrace$, allowing us to write
\[
\mathrm{e}^{- p \int_0^s \ind_{(-\infty,0)}(U^\pi_r) \mathrm{d}r} = \p_x \left( T^p_i \notin \left\lbrace r \in (0,s] \colon U^\pi_r < 0 \right\rbrace , \; \text{for all $i \geq 1$} \vert \mathcal{F}_\infty \right) = \e_x \left[ \ind_{\{\sigma_p^\pi > s\}} \vert \mathcal{F}_\infty \right] .
\]
Consequently,
\[
\e_x \left[ \sum_{0 < s < t} \mathrm{e}^{-q s - p \int_0^s \ind_{(-\infty,0)}(U^\pi_r) \mathrm{d}r} \left(\Delta L^\pi_s - \beta \ind_{\{\Delta L^\pi_s > 0\}} \right) \right] \\
= \e_x \left[ \sum_{0<s<\sigma_p^\pi \wedge t} \mathrm{e}^{-q s} \left(\Delta L^\pi_s - \beta \ind_{\{\Delta L^\pi_s > 0\}} \right) \right] ,
\]
where we used that $X$ and $L^\pi$ are both adapted to the filtration. 

Finally, since for all $x \in \reals$ we have $\left(\Gamma-q-p\ind_{(-\infty,0)}(x) \right) w (x) \leq 0$, then, using standard arguments, we obtain
\begin{multline*}
w(x) \geq \e_x \left[ \sum_{s>0} \mathrm{e}^{-q s - p \int_0^s \ind_{(-\infty,0)}(U^\pi_r) \mathrm{d}r} \left(\Delta L^\pi_s - \beta \ind_{\{\Delta L^\pi_s > 0\}} \right) \right] \\
= \e_x \left[ \sum_{0<s<\sigma_p^\pi} \mathrm{e}^{-q s} \left(\Delta L^\pi_s - \beta \ind_{\{\Delta L^\pi_s > 0\}} \right) \right] = v_\pi(x) .
\end{multline*}
\end{proof}

\section{A family of simple impulse strategies}

For fixed values $0 \leq a < b$,  the corresponding $(a,b)$-strategy is defined as follows: each time the controlled process $U^{a,b}$ crosses level $b$, a dividend payment of $b-a$ is made. Mathematically speaking, the corresponding dividend process is defined by: $L^{a,b}_{0+}=\left(X_0-a \right) \ind_{\{X_0 > b\}}$ and, for $t > 0$,
\[
L^{a,b}_t = L^{a,b}_{0+} + \sum_{k=2}^\infty (b-a) \ind_{\{T_k^{a,b}<t\}} ,
\]
where $T_k^{a,b} := \inf \left\lbrace t > 0 \colon X_t > X_0 \wedge b + (b-a)(k-1) \right\rbrace$, and then we set $U^{a,b}_t := X_t-L^{a,b}_t$. The performance function $v_{a,b}$ of this strategy can be expressed analytically using scale functions of the underlying SNLP.

First, recall that scale functions $\left\lbrace W^{(q)}, q \geq 0 \right\rbrace$ are equal to zero on $(-\infty,0)$ and are given by
$$
\int_0^\infty \mathrm{e}^{-\theta x} W^{(q)}(x) \mathrm{d}x = \left(\psi(\theta)-q \right)^{-1} ,
$$
for all $\theta> \Phi(q)=\sup \left\lbrace \lambda \geq 0 \colon \psi(\lambda)=q \right\rbrace$, on $[0,\infty)$.  Also, we need to define
\[
x \mapsto \Zq{x} = p \int_0^\infty \mathrm{e}^{-\Phi(p+q) y} W^{(q)}(x+y) \mathrm{d}y , \quad x \in \reals .
\]
Then, for $x > 0$, we have
\[
\Zqprime{x} = p \int_0^\infty \mathrm{e}^{-\Phi(p+q) y} W^{(q)\prime}(x+y) \mathrm{d}y ,
\]
which is well defined since $W^{(q)}$ is differentiable almost everywhere.

Now, the performance function of an admissible $(a,b)$-strategy can be easily computed using standard Markovian-type arguments and Parisian first-passage identities for $X$.

\begin{lemma}\label{L:performance}
For a fixed pair $(a,b)$ such that  $a \geq 0$ and $b > a + \beta$, we have
\[
v_{a,b} (x) =
\begin{cases}
\left( \frac{b-a-\beta}{\Zq{b}-\Zq{a}} \right) \Zq{x} & \text{for $x \in (-\infty,b]$,}\\
x-a-\beta + \left( \frac{b-a-\beta}{\Zq{b}-\Zq{a}} \right) \Zq{a} & \text{for $x \in (b,\infty)$.}
\end{cases}
\]
\end{lemma}
\begin{proof}
As the controlled process $U^{a,b}$ is allowed to go below zero, we modify the proof of Proposition~2 in \cite{loeffen_2009b} as follows. As in the proof of Proposition~1 in \cite{renaud_2019}, we define
\[
\kappa^p = \inf \left\lbrace t>0 \colon t-g_t > \mathbf{e}_p^{g_t} \; \text{and} \; X_t < 0 \right\rbrace ,
\]
where $g_t = \sup \left\lbrace 0 \leq s \leq t \colon X_s \geq 0 \right\rbrace$. This is the time of Parisian ruin with rate $p$ for $X$. Also, let us define the stopping time $\tau_b^+ = \inf \left\lbrace t>0 \colon X_t >b \right\rbrace$. Using Equation (16) in~\cite{lkabous-renaud_2019}, we can write, for $x \in (-\infty, b]$,
\[
\e_x \left[ \mathrm{e}^{-q \tau_b^+} \ind_{\{\tau_b^+<\kappa^p\}} \right] = \frac{\Zq{x}}{\Zq{b}} .
\]
Consequently, for $x \in (-\infty,b]$,  using the strong Markov property, we can write
\[
v_{a,b} (x) = \frac{\Zq{x}}{\Zq{b}} v_{a,b} (b) = \frac{\Zq{x}}{\Zq{b}} \left( b-a-\beta + v_{a,b} (a) \right) .
\]
When $x=a$, we can solve for $v_{a,b} (a)$ and the result follows.
\end{proof}

\begin{remark}
Note that $v_{a,b}$ is not necessarily continuously differentiable at $x=b$.
\end{remark}

The following lemma was proved in \cite{renaud_2019}.
\begin{lemma}\label{lemma:log-convexity}
If the tail of the Lévy measure is log-convex, then $\Zqprime{\cdot}$ is log-convex. In particular, $\Zqdoubleprime{\cdot}$ exists and is continuous almost everywhere.
\end{lemma}

Next, our objective is to find an optimal pair $(a^\ast,b^\ast)$, that is such that, for any other admissible $(a,b)$-strategy, we have $v_{a^\ast,b^\ast} (x) \geq v_{a,b} (x)$ for all $x \in \reals$.

\begin{lemma}
If the tail of the Lévy measure is log-convex, then there exists a unique optimal pair $(a^\ast,b^\ast)$. Also, we have that $a^\ast \wedge 0 \leq c^\ast < b^\ast$, where
\[
c^\ast := \sup \left\lbrace c \geq 0 \colon \Zqprime{c} \leq \Zqprime{x} , \; \text{for all $x \geq 0$} \right\rbrace \in [0,\infty) .
\]
\end{lemma}
\begin{proof}
Let us define the function
\[
g(x,y) = \frac{\Zq{y}-\Zq{x}}{y-x-\beta}
\]
on $\left\lbrace (x,y) \in \reals \times \reals \colon x \geq 0 , \, y>x+\beta \right\rbrace$. Under the assumption on the Lévy measure, thanks to Lemma~\ref{lemma:log-convexity}, we have that $\Zq{\cdot}$ is continuously differentiable. Consequently, the conclusion of Proposition~3 in \cite{loeffen_2009b} also holds in our setup. More precisely, there exists at least one pair $(\hat{a},\hat{b})$ in the domain of $g$ such that either $\hat{a}>0$ and $\Zqprime{\hat{a}}=\Zqprime{\hat{b}}$, or $\hat{a}=0$. Moreover, we have
\begin{equation}\label{E:identity-at-optimality}
\Zqprime{\hat{b}} = g(\hat{a},\hat{b}) ,
\end{equation}
in both cases. The proof of uniqueness goes also as in Section~4 of \cite{loeffen_2009b}, even with our relaxed condition on the Lévy measure. Indeed, it was proved in \cite{renaud_2019} that if the tail of the Lévy measure is log-convex, then $\Zqprime{\cdot}$ is strictly increasing on $(c^\ast,\infty)$ (in that paper, the notation $b_p^\ast$ is used instead of $c^\ast$). This, and the smoothness of $\Zq{\cdot}$ (cf.\ Lemma~\ref{lemma:log-convexity}), is all we need to conclude.
\end{proof}

\begin{remark}
Note that, under the assumptions of the previous lemma, the function $v_{a^\ast,b^\ast}$ is continuously differentiable at $x=b^\ast$.
\end{remark}

\section{An optimal strategy}

Here is our main result.

\begin{theorem}
If the tail of the Lévy measure is log-convex, then the impulse strategy given by $(a^\ast,b^\ast)$ is optimal and we have
\[
v_\ast (x) =
\begin{cases}
\frac{\Zq{x}}{\Zqprime{b^\ast}} & \text{for $x \in (-\infty,b^\ast]$,}\\
x-b^\ast + \frac{\Zq{a^\ast}}{\Zqprime{b^\ast}} & \text{for $x \in (b^\ast,\infty)$.}
\end{cases}
\]
\end{theorem}
\begin{proof}
Using~\eqref{E:identity-at-optimality} with $(a^\ast,b^\ast)$, we can simplify the expression of the performance function for the $(a^\ast,b^\ast)$-strategy, as given in Lemma~\ref{L:performance}.

It was proved in \cite{renaud_2019} that $\Gamma \Zq{x} - \left( q+p \ind_{(-\infty,0)}(x) \right) \Zq{x} = 0$ for all $x \in \reals$. Also, since $\Zq{\cdot}$ is continuously differentiable under the assumption on the Lévy measure, we can use the same arguments as in Lemma~5 in \cite{loeffen_2009b} to prove that $v_{a^\ast,b^\ast}(y)-v_{a^\ast,b^\ast}(x) \geq y-x-\beta$, for all $y \geq x \geq 0$.

The rest is standard due to the fact that, under the assumption on the Lévy measure, $\Zqprime{\cdot}$ is sufficiently smooth and strictly increasing on $(c^\ast,\infty)$.
\end{proof}

\section*{Acknowledgements}

Funding in support of this work was provided by a Discovery Grant from the Natural Sciences and Engineering Research Council of Canada (NSERC).

%
%
\bibliographystyle{abbrv}
\bibliography{references-SNLPs.bib}

\end{document}